\documentclass[12pt]{amsart}

\usepackage{amsmath}

\calclayout

\usepackage{amsfonts,amssymb,amsthm,enumitem}
\usepackage{mathtools}
\usepackage{nicefrac,setspace}
\usepackage{xcolor}

\usepackage[T1]{fontenc}

\usepackage{hyperref}       % hyperlinks
\usepackage{url}            % simple URL typesetting
\usepackage{booktabs}       % professional-quality tables
\usepackage{amsfonts}       % blackboard math symbols
\usepackage{nicefrac}       % compact symbols for 1/2, etc.
\usepackage{microtype}      % microtypography

\newtheorem*{theorem*}{Theorem}
\newtheorem*{claim*}{Claim}
\newtheorem*{lemma*}{Lemma}
\newtheorem*{remark}{Remark}

\newtheorem*{corollary*}{Corollary}

\newcommand{\X}{{\mathcal{X}}}

\newcommand{\N}{\mathbb{N}}
\newcommand{\eps}{\epsilon}

\newcommand{\R}{\mathbb{R}}

\newcommand{\E}{\mathop \mathbb{E}}

\newcommand{\cX}{\mathcal{X}}
\newcommand{\cY}{\mathcal{Y}}

\title{A note on shifting distributions via Poisson races}

\author{Amir Yehudayoff}
\address{Department of Computer Science, The University of Copenhagen and Department of Mathematics, Technion-IIT}
\email{amir.yehudayoff@gmail.com}

\begin{document}

\begin{abstract}
This expository note is about simulating a target distribution $Q$ from observations of 
a proposal distribution $P$.
In the model suggested by Harsha, Jain, McAllester and Radhakrishnan,
we observe an infinite sequence of i.i.d.\ samples $X_1,X_2,\ldots$ from $P$.
The goal is to find some index $I \in \{1,2,\ldots\}$ such that $X_I \sim Q$
while minimizing $\E \log I$.
Following the Poisson-race approach developed by
Maddison, Li and El Gamal, and others, this note 
shows that if $D(Q||P) < \infty$ then
there is a $P$ to $Q$ simulator $I$ such that
$\E [\log I] 
\leq D(Q||P) + 1.45 \|Q-P\|_1$.
In the other direction, for every $P$ to $Q$ simulator $I$,
the cost is at least
$\E[\log I] \geq \frac{1}{2} \max \{ D(Q||P),  \|Q-P\|_1 \}$.
\end{abstract}
\maketitle

\section{Introduction}

The problem of simulating a target distribution $Q$ 
given access to some other distribution $P$ has been studied for a long time.
It was studied by von Neumann~\cite{vonNeumann1951RandomDigits}
who introduced rejection sampling,
and by Wyner~\cite{Wyner1975CommonInformation} who introduced
the common information problem.
The specific form of the problem described below was first suggested,
to the best of my knowledge, in the work of
Harsha, Jain, McAllester and Radhakrishnan~\cite{HarshaEtAl2007Correlation}.
But this type of problem already implicitly appeared in the work of Hajek and Pursley~\cite{hajek1979evaluation}, 
and a close variant of it was termed the functional 
representation lemma by El Gamal and Kim~\cite{el2011network}
and the strong functional representation lemma by
Li and El Gamal~\cite{LiElGamal2018SFRL}.
This note is too brief to provide a comprehensive description of the history
(for more details, see the references above
and also~\cite{block2023sample,goc2024channel,liu2018rejection}
and references therein).

There are two distributions $P,Q$ over some finite set\footnote{We focus
on the case of a finite set for simplicity.} $\cX$.
We observe an infinite sequence of i.i.d.\ samples $X_1,X_2,\ldots$
from $P$.
Our goal is to find some index $I \in \{1,2,\ldots\}$ such that $X_I \sim Q$;
that is, $\Pr[X_I = x] = Q(x)$ for every $x \in \cX$.
We call such an index $I$ a simulator (it can also be randomized).

Here is an important example:
a two-player game with public shared randomness.
Alice knows some distribution $R$ on pairs $\cX \times \cY$,
and Bob knows the marginal distribution $R_\cY$.
Alice gets as input $X \sim R_\cX$ and
her goal is to communicate some information to Bob that will allow
Bob to compute $Y$ with the property that $(X,Y) \sim R$.
The two extreme cases are (1) when $R = R_\cX \times R_\cY$ is a product distribution
and (2) when $\cX = \cY$ and $R$ is supported on the diagonal $\{(x,x)\}$.
The players can use a simulator to solve this problem.
The players interpret the shared randomness as infinitely many i.i.d.\ samples
$Y_1,Y_2,\ldots$ from $P:=R_\cY$.
Alice gets $X$ and constructs a simulator $I$ for $Q( y ): = R(y|X)$.
She sends $I$ to Bob who can safely output $Y: = Y_I$.
In case (1) above, Alice just sends $I=1$.
In case (2), Alice needs to send the entire $X$ to Bob
which requires about the entropy of $X$ many bits.

Most of the works mentioned above were interested in minimizing
the expected length of prefix-free encodings of simulators.
This note considers a different (but closely related) minimization problem.
The goal here is to minimize the ``description length'' of $I$ in bits.
The cost of a simulator $I$ is 
\[cost(I):=\E [\log I]\] where $\log$ is base two.

The main purpose of this text is to prove the following theorem,
which provides both an upper bound
and lower bounds on the cost of simulators. 

\begin{theorem*}
Let $P,Q$ be two distributions over the same finite set.
If $D(Q||P) < \infty$ then
there is a $P$ to $Q$ simulator $I$ such that
\begin{align*}
\E [\log I] 
\leq D(Q||P) + 1.45 \|Q-P\|_1.
\end{align*}
On the other hand, for every $P$ to $Q$ simulator $I$,
\[\E[\log I] \geq \frac{1}{2} \max \{ D(Q||P),  \|Q-P\|_1 \}.\]
In particular, if $D(Q||P)=\infty$ then there is no $P$ to $Q$ simulator.
\end{theorem*}

Harsha, Jain, McAllester and Radhakrishnan~\cite{HarshaEtAl2007Correlation}
showed that for every $P,Q$ such that $D(Q||P) < \infty$,
there is a simulator $I$ with cost
$\E [\log I] \leq D(Q||P) + C$ where $D(Q||P) = \sum_x Q(x) \log \frac{Q(x)}{P(x)}$ is the Kullback–Leibler divergence
and $C>0$ is some universal constant.
The mechanism behind the proof presented here is different from the mechanism in~\cite{HarshaEtAl2007Correlation}.
Their simulator is based on a carefully constructed rejection sampling
procedure. 
This text uses the Poisson-race mechanism described by Maddison~\cite{Maddison2016PoissonMonteCarlo} and that also appeared in 
the work of Li and El Gamal~\cite{LiElGamal2018SFRL} mentioned above. 
The proof mechanism is informative and it may, e.g., allow to control higher moments of
$\log I$.

The upper bound in the theorem comprises two terms.
The divergence term $D(Q||P)$ is dominant when the distributions are far away.
The total variation term $1.45 \|Q-P\|_1$ is dominant when the distributions are close.
The total variation term could be as large as $+2.9$,
but it is known that it can be replaced by $+1$;
see~\cite{li2024pointwise} (a simple variant of the proof presented here
also leads to this $+1$ bound).
The theorem above is strictly stronger than 
all the results mentioned above when $P,Q$ are close
(because all of them had some additive constant term).
To the best of my knowledge, the two lower bounds were not observed before.

The lower bound has an interesting consequence.
An upper bound of the form
$cost(I) \leq C D(Q||P)$ is not achievable in general (where $C>0$ is any universal constant).
For example, when $P$ is an unbiased coin and $Q$ is a coin with bias $\eps>0$,
we have that $D(Q||P) \approx \eps^2$ but $\E \log I \gtrsim \|Q-P\|_1 \approx \eps$.
In this example, for every $I$, the cost of $I$ is much larger than $D(Q||P)$.

The authors of~\cite{HarshaEtAl2007Correlation} showed a lower bound of $D(Q||P)$ 
on the cost of a prefix-free encoding of $I$.
This means that when $D(Q||P)$ is large, the lower bound 
stated above is essentially off by a factor of $2$.
The lower bound
$\E[\log I] \geq D(Q||P)$, however, is not true in general.
For example, if $P$ is a distribution on $\{0,1\}$ with $P(0)=1-\eps$
and $Q$ is defined by $Q(0)=1$ then 
$D(Q||P) = \frac{\eps}{\ln 2} + O(\eps^2)$ and if $I$ is the first occurrence of $0$
then $I$ is a simulator and $\E[\log I] = \eps + O(\eps^2)$
so that $cost(I) \leq 0.7 D(Q||P)$.

Here is an intuitive sketch of the mechanism.
A Poisson process can be thought of as a cloud of random points on 
the line $\R$ (see detailed description and references below). We can (independently) mark each point in the cloud by $x \in \cX$ using $P$.
So now we have a cloud of marked points. 
Each mark $x \in \cX$ has density $P(x)$ in the cloud.
This marked cloud represents the distribution $P$. 
Our goal is to shift from $P$ to $Q$. 
The main point is that Poisson processes are determined by their local densities.
Changing densities is straitforward. For example, if we dilate all points marked by $x$
by a factor of two then we halved their density.
More interestingly, if we dilate all points marked $x$ by
$\frac{P(x)}{Q(x)}$, then their density becomes \[P(x) \cdot \frac{Q(x)}{P(x)} = Q(x).\]
The new dilated cloud represents~$Q$.
We thus shifted from $P$ to $Q$ via a dilation.

\subsection*{Acknowledgement}

This note was written following the summer school 
A-PIC 2026: Probability, Information, Computing held at the African Institute for Mathematical Science in Kigali, Rwanda.
I wish to thank Jan H\k{a}z\l{}a, as well as the participants of the school and the AIMS staff, for their
 hospitality. 
 I also wish to thank A.\ El Gamal and C.T.\ Li for helpful comments.
 A preliminary draft of this note was written 
as a preparation for the school with the assistance of ChatGPT.

\section{Poisson processes}

Here is a very brief introduction to Poisson processes and Poisson races;
see for example~\cite{DaleyVereJones2003,LastPenrose2017,Kingman1993}.
A rate-$\lambda$ Poisson process on $\R_+$ 
\[0< T_1 < T_2 < \ldots\]
is characterized by the following properties. 
First, for every measurable $M \subset (0,\infty)$ of finite measure $|M|<\infty$,
the count
\[N(M) = \big| \{ i \in \N : T_i \in M\} \big| \]
is distributed as
\[N(M) \sim Pois(\lambda |M|).\]
Second, if $M_1,\ldots,M_m$ are disjoint then the counts
$N(M_1),\ldots,N(M_m)$ are independent. 

We mention a couple of basic properties.
If we multiply every $T_i$ by some fixed $\alpha > 0$
then we get a Poisson process with rate $\frac{\lambda}{\alpha}$.
The union of (finitely many) independent Poisson processes is a Poisson process 
with rate being the sum of the rates.

A key idea for simulation is to use Poisson marking.
Assume $P$ is a distribution on a finite domain $\cX$
and let $X_1,X_2,\ldots$ be i.i.d.\ samples from $P$.
Let $0 < T_1 < T_2 < \ldots$ be a rate-$1$ Poisson process
(chosen independently from $X_1,X_2,\ldots$).
For each $x \in \cX$, we can consider the marked subprocess
$T^x_1 < T^x_2 < \ldots$ that is obtained by keeping only
the $i$'s such that $X_i = x$. This marked process is a Poisson process
with rate $P(x)$.
In addition, these $|\cX|$ processes are independent
(if $P(x)=0$, the corresponding process is empty).

\section{The simulator}

\subsection{A Poisson race}

The following lemma summarizes the simulator and its correctness.

\begin{lemma*}
Let $P,Q$ be two distributions over a finite set $\cX$
such that $D(Q||P) < \infty$.
%$P(x) Q(x) > 0$ for all $x \in \cX$.
Let $X_1,X_2,\ldots$ be i.i.d.\ from $P$.
Let $T_1 < T_2 < \ldots$
be a rate-$1$ Poisson process on $(0,\infty)$ that is
independent of $X_1,X_2,\ldots$.
For each $i$, define
\[S_i = \frac{T_i}{R(X_i)}\]
where
\[R(x) := \frac{Q(x)}{P(x)}>0\]
(when $Q(X_i)=0$, define $S_i=\infty$).
Let
\[I = \mathsf{argmin} \ S_i.\]
Then,
\[X_I \sim Q.\]
\end{lemma*}

\begin{proof}
We have $|\X|$ independent marked Poisson processes $(T^x_i)$.
Because $D(Q||P) < \infty$,
if $Q(x)>0$ then $P(x) > 0$.
If $Q(x)=0$ and $P(x)>0$ then $S^x_i=\infty$ for all $i$.
If $P(x)=0$ then $x$ never appears.
So we can ignore the $x$'s such that $Q(x)P(x) =0$;
they are either fixed or never appear. 
For a remaining $x$, the process $(T^x_i)$ has rate $P(x)$ so 
$S^x_i$ is a Poisson process with rate
$P(x) R(x)= Q(x)$.

The processes $(S^x_i)$ are independent.
Their union is a marked Poisson process with rate $1 = \sum_x Q(x)$.
The union process is marked by $Q$. 
Order the elements of the new process as
$S_{(1)} < S_{(2)} < \ldots$
then the mark of $S_{(1)}$ is distributed according to $Q$.
The marks of $S_{(1)}$ and $T_I$ are the same by definition so 
\[X_I \sim Q.\]
The winner $I$ is well-defined because we need to consider only finitely many
of the $X_i$'s (only the first one of each type $x \in \cX$).
\end{proof}

\subsection{Upper bound on cost}
It remains to bound $\E \log I$ from above.
First, consider the regular conditional law given $\{X_I=x,S_I = s\}$.
The conditional $\{S_I=s\}$ means that 
$S_{(2)} < S_{(3)} < \ldots$ is a rate 1 Poisson process on $(s,\infty)$.
The conditional $\{X_I=x\}$ means that $S_{(1)}$ is marked by $x$
and the rest of the marks are independent of that choice.

In particular, for each $y \neq x$,
the process $(S^y_i)$ conditioned on $\{X_I=x,S_I=s\}$
is a Poisson process on $(s,\infty)$ with rate $Q(y)$.
This means that $(T_i^y) = (R(y) S_i^y)$ conditioned on $\{X_I=x,S_I=s\}$ is a  
Poisson process with rate $P(y)$ on $(s R(y),\infty)$.
Conditioned on $\{X_I=x,S_I=s\}$,
the expected number of points of $(T^y_i)$ before 
$T_I = s R(x)$ is therefore $s (R(x)-R(y))_+ P(y)$.

Summing over $y$,
\[\E[I-1|X_I=x,S_I=s] = \sum_y s P(y) (R(x)-R(y))_+ .\]
Because the rate of the process defined by the $S_i$'s is $1$ and it is independent of the markings,
\[\E[I|X_I=x] = \E_s \Big[ 1+ \sum_y s P(y) (R(x)-R(y))_+ \Big]
=  1+ \sum_y  P(y) (R(x)-R(y))_+ .\]
Because
\begin{align*}
& \Big(\sum_y  P(y) (R(x)-R(y))_+ \Big) - \Big( \sum_y  P(y) (R(y)-R(x))_+ \Big)\\
& = \sum_y  P(y) (R(x)-R(y)) \\
& = R(x) - \sum_y  P(y) \frac{Q(y)}{P(y)} = R(x) - 1,
\end{align*}
we can write
\begin{align*}
1+ \sum_y  P(y) (R(x)-R(y))_+
& = R(x) + \Big( \sum_y  P(y) (R(y)-R(x))_+ \Big) .
\end{align*}

\begin{remark}
Because $X_I \sim Q$,
this allows to control the expected value of $I$ as
\[\E[I] =\sum_x Q(x)  \Big( R(x) +  \sum_y  P(y) (R(y)-R(x))_+ \Big)
\leq 1+\sum_x \frac{Q^2(x)}{P(x)}  .\]
\end{remark}

Averaging over $X_I \sim Q$ as well, because $\log \xi$ is concave,
\begin{align*}
\E [\log I] 
& = \E_{x \sim Q} \E [\log I|x]  
 \leq \E_{x} \log( \E [ I|x])  .
\end{align*}
Substituting the above,
\begin{align*}
\E [\log I] 
& \leq \E_{x} \log\Big( R(x) + \Big( \sum_y  P(y) (R(y)-R(x))_+ \Big) \Big)  \\
& = D(Q||P)+
\Big( \sum_x Q(x)  \log\Big( 1 + \Big( \sum_y  P(y) \frac{(R(y)-R(x))_+}{R(x)} \Big) \Big) .
\end{align*}
It remains to bound the second term.
Because $\log_2(1+\xi) \leq \frac{\xi}{\ln 2}$ for $\xi \geq 0$,
\begin{align*}
&  \sum_x Q(x)  \log\Big( 1 + \Big( \sum_y  P(y) \frac{(R(y)-R(x))_+}{R(x)} \Big) \\
& \leq \frac{1}{\ln 2} \sum_{x,y} Q(x)   P(y) \frac{(R(y)-R(x))_+}{R(x)} \\
%& = \frac{1}{\ln 2} \sum_{x,y} Q(x)  P(y) \frac{(R(y)-R(x))_+}{Q(x)/P(x)} \\
& = \frac{1}{\ln 2} \sum_{x,y} P(x)  P(y) (R(y)-R(x))_+.
\end{align*}
By ``symmetry'' between $x,y$,
\begin{align*}
\sum_{x,y} P(x)  P(y) (R(y)-R(x))_+
& =\frac{1}{2}  \sum_{x,y} P(x)  P(y) |R(y)-R(x)| \\
& =\frac{1}{2}  \sum_{x,y} P(x)  P(y) |R(y)-1 + 1-R(x)| \\
& \leq \sum_{x,y} P(x)  P(y) |R(x)-1| 
%\\ & = \sum_{x} P(x)  |R(x)-1| 
= \|Q-P\|_1 .
\end{align*}
Altogether,
\begin{align*}
\E [\log I] 
\leq D(Q||P) + \frac{\|Q-P\|_1}{\ln 2}
\end{align*}

\section{Lower bounds}

Assume $X_1,X_2,\ldots$ are i.i.d.\ from $P$
and that $I$ is such that $X_I \sim Q$.

\subsection{Total variation bound}
For each $x$,
\[\Pr[X_I = x, I=1] \leq \Pr[X_1 = x] = P(x)\]
and 
\[\Pr[X_I = x, I=1] \leq \Pr[X_I=x] = Q(x).\]
So,
\[\Pr[I=1] \leq \sum_x \min\{P(x),Q(x)\} = 1 - \frac{\|Q-P\|_1}{2} . \]
If $I > 1$ then $\log I \geq 1$ so
\begin{align*}
\E [\log I]
\geq \Pr[I > 1] \geq \frac{\|Q-P\|_1}{2} .
\end{align*}

\subsection{Divergence bound}
For every $x$ such that $Q(x)>0$,
\[Q(x) \Pr[I=i|X_I=x] = \Pr[X_I=x,I=i]
\leq \Pr[X_i=x] \leq P(x)  \]
or
\[\Pr[I=i|X_I=x] \leq \frac{1}{R(x)} \]
where $R(x) = \frac{Q(x)}{P(x)}$.

If $R(x) \leq 1$ then
\[\E[\log I|X_I=x] \geq 0 \geq \frac{1}{2} \log R(x).\]
If $R(x) > 1$ then argue as follows.
Out of all distributions of $Y$ on $\{1,2,\ldots\}$
such that $\Pr[Y=y] \leq \frac{1}{R}$ for all $y$,
the one that minimizes $\E \log Y$
gives mass $\frac{1}{R}$ for the first $M:=\lfloor R \rfloor \geq 1$ elements
and the rest of the mass $1-\frac{M}{R} \leq \frac{1}{R}$ is given to the element $M+1$.
For this distribution,
\begin{align*}
\E[ \log Y]
& =\Big( \sum_{i \leq M} \frac{1}{R} \log i \Big)
+ \Big( 1- \frac{M}{R} \Big) \log (M+1)  \\
& = \frac{ \log (M!) + (R-M) \log(M+1) }{R} .
\end{align*}
Write $R=M+\theta$ with $0\le\theta<1$. 
For an integer $n$, we have
$\log(n!) \geq \frac{n}{2} \log n$
because $k(n+1-k) \geq n$.
So,
\begin{align*}
\log (M!) + (R-M) \log(M+1)  
& = \log (M!) + \theta \log(M+1)  \\
& = (1-\theta)\log (M!) + \theta \log((M+1)!) \\
& \geq (1-\theta) \frac{M}{2} \log(M) + \theta \frac{M+1}{2} \log(M+1) .
\end{align*}
Because $\xi \log \xi$ is convex,
\begin{align*}
& \log (M!) + (R-M) \log(M+1)  \\
& \geq \frac{1}{2} \big( (1-\theta)M + \theta (M+1) \big)
\log \big( (1-\theta)M + \theta (M+1) \big)  = \frac{1}{2} R \log R.
\end{align*}
We can conclude that also if $R(x) > 1$,
\[\E[\log I|X_I=x] \geq \frac{1}{2} \log R(x).\]

Taking expectation,
\[\E[ \log I] =
\E_{x \sim Q} \E [ \log I |x] \geq \frac{1}{2} \E \log R(x)
= \frac{D(Q||P)}{2}.\]

\bibliographystyle{amsalpha} 
\bibliography{Corr.bib}

\end{document}